\documentclass[12pt]{amsart}
\usepackage{amssymb}

\newtheorem{theorem}{Theorem}[section]
\newtheorem{definition}[theorem]{Definition}
\newtheorem{proposition}[theorem]{Proposition}

\begin{document}

\title[Transitive quantum groups]{Higher transitive quantum groups: theory and models}

\author{Teodor Banica}
\address{T.B.: Department of Mathematics, University of Cergy-Pontoise, F-95000 Cergy-Pontoise, France. {\tt teo.banica@gmail.com}}

\begin{abstract}
We investigate the notion of $k$-transitivity for the quantum permutation groups $G\subset S_N^+$, with a brief review of the known $k=1,2$ results, and with a study of what happens at $k\geq3$. We discuss then matrix modelling questions for the algebras $C(G)$, notably by introducing the related notions of double and triple flat matrix model. At the level of the examples, our main results concern the quantum groups coming from the complex Hadamard matrices, and from the Weyl matrices.
\end{abstract}

\subjclass[2010]{46L65 (60B15)}
\keywords{Quantum permutation, Matrix model}

\maketitle

\section*{Introduction}

One interesting question is that of finding matrix models for the coordinates of various quantum groups, or more general noncommutative manifolds. The case of the quantum permutation groups $G\subset S_N^+$ is of particular interest, due to the rich structure of such quantum groups, and to the variety of modelling methods which can be employed.

Such questions have been systematically investigated in the last years, starting with \cite{ban}, \cite{bne}, with some key advances being recently obtained in \cite{bch}, \cite{bfr}. The present paper continuates this work. According to \cite{bch}, \cite{bne}, at the ``center'' of this modelling theory are the flat models $\pi:C(G)\to M_N(C(X))$ for the transitive subgroups $G\subset S_N^+$. We will refine here this point of view, by taking into account the higher transitivity properties of $G$. At the level of the examples, our main results will concern the quantum groups coming from the complex Hadamard matrices, and from the Weyl matrices.

The paper is organized as follows: sections 1-2 contain various preliminaries and generalities, in 3-4 we discuss the notion of higher transitivity for the quantum groups, and in 5-6 we discuss random matrix models, and the notion of higher flatness for them.

\medskip

\noindent {\bf Acknowledgements.} I would like to thank A. Chirvasitu for useful discussions.
 
\section{Quantum permutations}

We recall that a magic unitary matrix is a square matrix $u\in M_N(A)$ over an abstract $C^*$-algebra, all whose entries are projections ($p=p^*=p^2$), summing up to 1 on each row and each column. The following key definition is due to Wang \cite{wa2}:

\begin{definition}
The quantum permutation group $S_N^+$ is the abstract spectrum, in the sense of the general $C^*$-algebra theory, of the universal algebra
$$C(S_N^+)=C^*\left((u_{ij})_{i,j=1,\ldots,N}\Big|u={\rm magic}\right)$$
with compact quantum group structure coming from the maps
$$\Delta(u_{ij})=\sum_ku_{ik}\otimes u_{kj}\quad,\quad\varepsilon(u_{ij})=\delta_{ij}\quad,\quad S(u_{ij})=u_{ji}$$
which are respectively the comultiplication, counit and antipode of $C(S_N^+)$.
\end{definition}

As a first remark, the algebra $C(S_N^+)$ is indeed well-defined, because the magic condition forces $||u_{ij}||\leq1$, for any $C^*$-norm. By using the universal property of this algebra, we can define maps $\Delta,\varepsilon,S$ as above. We obtain in this way a Hopf $C^*$-algebra in the sense of Woronowicz \cite{wo1}, \cite{wo2}, satisfying the extra assumption $S^2=id$. Thus, according to \cite{wo1}, \cite{wo2}, the abstract spectrum $S_N^+$ is indeed a compact quantum group. See \cite{wa2}.

The terminology comes from the following result, also from \cite{wa2}:

\begin{proposition}
The quantum permutation group $S_N^+$ acts on the set $X=\{1,\ldots,N\}$, the corresponding coaction map $\Phi:C(X)\to C(S_N^+)\otimes C(X)$ being given by:
$$\Phi(\delta_i)=\sum_ju_{ij}\otimes\delta_j$$
In addition, $S_N^+$ is in fact the biggest compact quantum group acting on $X$.
\end{proposition}

\begin{proof}
Given a compact quantum group $G$ in the sense of Woronowicz \cite{wo1}, \cite{wo2}, it is straightforward to check that the formula $\Phi(\delta_i)=\sum_ju_{ij}\otimes\delta_j$ defines a morphism of algebras, which is in addition a coaction map, precisely when the matrix $u=(u_{ij})$ is a magic corepresentation of $C(G)$. But this gives all the assertions. See \cite{wa1}, \cite{wa2}.
\end{proof}

In practice, it is useful to think of $S_N^+$ as being a ``liberation'' of $S_N$, viewed as an algebraic group, $S_N\subset O_N$, via the standard permutation matrices, as follows:

\begin{proposition}
We have a quantum group embedding $S_N\subset S_N^+$, whose transpose quotient map $C(S_N^+)\to C(S_N)$ maps standard coordinates to standard coordinates,
$$u_{ij}\to\chi\left(\sigma\in S_N\Big|\sigma(j)=i\right)$$
and whose kernel is the commutator ideal of $C(S_N^+)$. This embedding is an isomorphism at $N=2,3$, but not at $N\geq4$, where $S_N^+$ is both non-classical, and infinite.
\end{proposition}

\begin{proof}
Observe first that the characteristic functions $\chi_{ij}$ in the statement are indeed the standard coordinates on $S_N\subset O_N$. These functions from a magic matrix, so we have a morphism as in the statement, given by $u_{ij}\to\chi_{ij}$. This morphism is surjective by Stone-Weierstrass, so the corresponding transpose map is an embedding $S_N\subset S_N^+$.

The fact that this embedding $S_N\subset S_N^+$ is compatible with the quantum group structures follows from definitions. In fact, this comes as well from Proposition 1.2 above, because $S_N$ does act on the set $X$ there. Finally, the assertion about the commutator ideal, which tells us that we are in a ``liberation'' situation, $(S_N^+)_{class}=S_N$, follows either directly, of from Proposition 1.2, because $S_N$ is the biggest classical group acting on $X$.

Regarding now the last assertion, this is clear at $N=2$, because the entries of a $2\times2$ magic unitary automatically commute. At $N=3$ the commutation check is more tricky, and the result was proved in \cite{wa2}. As for the assertion at $N\geq4$, this comes by using magic matrices of type $u=diag(v,w,1_{N-4})$, with $v,w$ being suitable $2\times2$ magic matrices. Indeed, we can choose $v,w$ such that the algebra generated by their coefficients is noncommutative, and infinite dimensional, and this gives the result. See \cite{wa2}.
\end{proof}

The main results regarding $S_N^+$ can be summarized as follows:

\begin{theorem}
The quantum permutation groups $S_N^+$ are as follows:
\begin{enumerate}
\item The Tannakian dual of $S_N^+$, formed by the spaces $C_{kl}=Hom(u^{\otimes k},u^{\otimes l})$, is the Temperley-Lieb algebra of index $N$, and this, at any $N\in\mathbb N$.

\item Equivalently, $C_{kl}$ appears as the span of the space $NC(k,l)$ of noncrossing partitions between an upper row of $k$ points, and a lower row of $l$ points.

\item The main character $\chi=\sum_iu_{ii}$ follows a Marchenko-Pastur (or free Poisson) law at any $N\geq4$, with respect to the Haar integration over $S_N^+$.

\item The fusion rules for the representations of $S_N^+$ coincide with the Clebsch-Gordan rules for the representations of $SO_3$, and this, at any $N\geq4$.

\item We have $S_4^+=SO_3^{-1}$, and the quantum subgroups $G\subset S_4^+$ are subject to an ADE classification, similar to the McKay classification of the subgroups of $SO_3$.

\item The quantum groups $S_N^+$ with $N\geq5$ are not coamenable, and the dimensions of their irreducible representations are bigger than those of $SO_3$.
\end{enumerate}
\end{theorem}

\begin{proof}
There are many results here, and for full explanations and references on this material, we refer to the survey paper \cite{bbc}. Let us just mention here that:

(1) This follows from the fact that the multiplication and unit of any algebra, and in particular of $\mathbb C^N$, can be modelled by $m=|\cup|$ and $u=\cap$, which generate $TL$.

(2) This is a useful version of (1), with the underlying isomorphism $NC_2(2k,2l)\simeq NC(k,l)$ being obtained by fattening/collapsing neighbors.

(3) This follows from (1) or (2), because the moments of $\chi$ count the number of diagrams in $NC_2(2k)\simeq NC(k)$, well-known to be the Catalan numbers $C_k=\frac{1}{k+1}\binom{2k}{k}$.

(4) This is a key result, which follows from (1) or (2), or directly from (3), the idea being that the Catalan numbers correspond to the Clebsch-Gordan rules.

(5) This is something technical, comes from the linear isomorphism $\mathbb C^4\simeq M_2(\mathbb C)$, the quantum symmetry groups of these latter algebras being $S_4^+,SO_3$.

(6) Here the dimension claim is clear, because $4$ gets replaced by $N\geq5$. As for the amenability claim, this follows from (3), and from the Kesten criterion.
\end{proof}

Generally speaking, we refer to the survey paper \cite{bbc} for further details regarding the above results, and for other known things about the quantum permutation groups. In what follows, we will come back to all this, later on, with some more details.

\section{Orbits and orbitals}

We are interested in what follows in the notions of orbits, orbitals and higher orbitals for the quantum subgroups $G\subset S_N^+$. Let us begin with the orbits. The study here goes back to Bichon's paper \cite{bi2}, which contains (implicitly) the definition of the orbits, plus a key result, namely the calculation of these orbits for the group duals $G=\widehat{\Gamma}\subset S_N^+$.

The systematic study of such orbits, and of the related notion of quantum transitivity, was started later on, in the recent paper \cite{bfr}. Let us begin with:

\begin{definition}
Let $G\subset S_N^+$ be a closed subgroup, with magic unitary $u=(u_{ij})$, and consider the equivalence relation on $\{1,\ldots,N\}$ given by $i\sim j\iff u_{ij}\neq0$.
\begin{enumerate}
\item The equivalence classes under $\sim$ are called orbits of $G$.

\item $G$ is called transitive when the action has a single orbit. 
\end{enumerate}
In other words, we call a subgroup $G\subset S_N^+$ transitive when $u_{ij}\neq0$, for any $i,j$.
\end{definition}

Here the fact that $\sim$ as defined above is indeed an equivalence relation follows by applying $\Delta,\varepsilon,S$ to a formula of type $u_{ij}\neq0$. For details, see \cite{bfr}.

Generally speaking, the theory from the classical case extends well to the quantum setting, and we have in particular the following result:

\begin{proposition}
For a subgroup $G\subset S_N^+$, the following are equivalent:
\begin{enumerate}
\item $G$ is transitive.

\item $Fix(u)=\mathbb C\xi$, where $\xi$ is the all-one vector.

\item $\int_Gu_{ij}=\frac{1}{N}$, for any $i,j$.
\end{enumerate}
\end{proposition}

\begin{proof}
Here $(1)\implies(2)$ follows from \cite{bi2}, $(2)\implies(3)$ follows by using the general theory in \cite{wo1}, and $(3)\implies(1)$ is trivial. For details here, we refer to \cite{bch}.
\end{proof}

Let us discuss now the notion of orbital. The definition here, which goes back to the recent papers \cite{lmr}, \cite{mrv}, is quite similar to Definition 2.1 above, as follows:

\begin{definition}
Let $G\subset S_N^+$ be a closed subgroup, with magic unitary $u=(u_{ij})$, and consider the equivalence relation on $\{1,\ldots,N\}^2$ given by $(i,k)\sim (j,l)\iff u_{ij}u_{kl}\neq0$.
\begin{enumerate}
\item The equivalence classes under $\sim$ are called orbitals of $G$.

\item $G$ is called doubly transitive when the action has two orbitals. 
\end{enumerate}
In other words, we call $G\subset S_N^+$ doubly transitive when $u_{ij}u_{kl}\neq0$, for any $i\neq k,j\neq l$.
\end{definition}

Onca again, the fact that we have indeed an equivalence relation comes from a straightforward computation, performed in \cite{lmr}. It is clear from definitions that the diagonal $D\subset\{1,\ldots,N\}^2$ is an orbital, and that its complement $D^c$ must be a union of orbitals. With this remark in hand, the meaning of (2) is that the orbitals must be $D,D^c$.

Among the other basic results established in \cite{lmr} is the fact, analogous to the above-mentioned result from \cite{bi1} regarding the orbits, that, with suitable definitions, the space $Fix(u^{\otimes 2})$ consists of the functions which are constant on the orbitals.

In analogy with Proposition 2.2 above, we have:

\begin{theorem}
For a doubly transitive subgroup $G\subset S_N^+$, we have:
$$\int_Gu_{ij}u_{kl}=\begin{cases}
\frac{1}{N}&{\rm if}\ i=k,j=l\\
0&{\rm if}\ i=k,j\neq l\ {\rm or}\ i\neq k,j=l\\
\frac{1}{N(N-1)}&{\rm if}\ i\neq k,j\neq l
\end{cases}$$
Moreover, this formula characterizes the double transitivity.
\end{theorem}

\begin{proof}
We use the standard fact, from \cite{wo1}, that the integrals in the statement form the projection onto $Fix(u^{\otimes 2})$. Now if we assume that $G$ is doubly transitive, $Fix(u^{\otimes 2})$ has dimension 2, and therefore coincides with $Fix(u^{\otimes 2})$ for the usual symmetric group $S_N$. Thus the integrals in the statement coincide with those for the symmetric group $S_N$, which are given by the above formula. Finally, the converse is clear as well.
\end{proof}

Regarding now the examples, the available constructions of transitive quantum groups were surveyed in \cite{bch}. As a first class of examples, we have the quantum automorphism groups of various finite graphs \cite{bbc}, \cite{bi1}, \cite{cha}, \cite{swe}, whose fine transitivity properties were recently studied in \cite{lmr}, \cite{mrv}. We have as well a second class of examples, coming from various matrix model constructions, that we will discuss later on.

\section{Permutation groups}

In this section and in the next one we discuss the notion of $k$-transitivity, at $k\in\mathbb N$. We begin our study by recalling a few standard facts regarding the symmetric group $S_N$, and its subgroups $G\subset S_N$, from a representation theory/probabilistic viewpoint. Generally speaking, we refer to \cite{bco} for symmetric group material of this type. In the next section, where we will deal with quantum groups, our reference will be \cite{bco} as well.

We first have the following standard result:

\begin{proposition}
Consider the symmetric group $S_N$, together with its standard matrix coordinates $u_{ij}=\chi(\sigma\in S_N|\sigma(j)=i)$. We have the formula
$$\int_{S_N}u_{i_1j_1}\ldots u_{i_kj_k}=\begin{cases}
\frac{(N-|\ker i|)!}{N!}&{\rm if}\ \ker i=\ker j\\
0&{\rm otherwise}
\end{cases}$$
where $\ker i$ denotes as usual the partition of $\{1,\ldots,k\}$ whose blocks collect the equal indices of $i$, and where $|.|$ denotes the number of blocks.
\end{proposition}

\begin{proof}
According to the definition of $u_{ij}$, the integrals in the statement are given by:
$$\int_{S_N}u_{i_1j_1}\ldots u_{i_kj_k}=\frac{1}{N!}\#\left\{\sigma\in S_N\Big|\sigma(j_1)=i_1,\ldots,\sigma(j_k)=i_k\right\}$$

Since the existence of $\sigma\in S_N$ as above requires $i_m=i_n\iff j_m=j_n$, this integral vanishes when $\ker i\neq\ker j$. As for the case $\ker i=\ker j$, if we denote by $b\in\{1,\ldots,k\}$ the number of blocks of this partition, we have $N-b$ points to be sent bijectively to $N-b$ points, and so $(N-b)!$ solutions, and the integral is $\frac{(N-b)!}{N!}$, as claimed.
\end{proof}

We recall now that each action $G\curvearrowright\{1,\ldots,N\}$ produces an action $G\curvearrowright\{1,\ldots,N\}^k$ for any $k\in\mathbb N$, and by restriction, $G$ acts on the following set:
$$I_N^k=\left\{(i_1,\ldots,i_k)\in\{1,\ldots,N\}^k\Big|i_1,\ldots,i_k\ {\rm distinct}\right\}$$

We have the following well-known result:

\begin{theorem}
Given a subgroup $G\subset S_N$, with standard matrix coordinates denoted $u_{ij}=\chi(\sigma|\sigma(j)=i)$, and a number $k\leq N$, the following are equivalent:
\begin{enumerate}
\item $G$ is $k$-transitive, in the sense that $G\curvearrowright I_N^k$ is transitive.

\item $Fix(u^{\otimes k})$ is minimal, i.e. is the same as for $G=S_N$.

\item $\dim Fix(u^{\otimes k})=B_k$, where $B_k$ is the $k$-th Bell number.

\item $\int_Gu_{i_1j_1}\ldots u_{i_kj_k}=\frac{(N-k)!}{N!}$, for any $i,j\in I_N^k$.

\item $\int_Gu_{i_1j_1}\ldots u_{i_kj_k}\neq0$, for any $i,j\in I_N^k$.

\item $u_{i_1j_1}\ldots u_{i_kj_k}\neq0$, for any $i,j\in I_N^k$.
\end{enumerate}
\end{theorem}

\begin{proof}
All this is well-known, the idea being as follows:

$(1)\implies(2)$ This follows from the fact that $u^{\otimes k}$ comes by summing certain actions $G\curvearrowright I_N^r$ with $r\leq k$, and the transitivity at $k$ implies the transitivity at any $r\leq k$.

$(2)\implies(3)$ This comes from the well-known fact that for the symmetric group $S_N$, the multiplicity $\#(1\in u^{\otimes k})$ equals the $k$-th Bell number $B_k$, for any $k\leq N$.

$(3)\implies(4)$ We can use the fact that $P_{i_1\ldots i_k,j_1\ldots j_k}=\int_Gu_{i_1j_1}\ldots u_{i_kj_k}$ is the orthogonal projection onto $Fix(u^{\otimes k})$. Thus we can assume $G=S_N$, and here we have:
$$\int_{S_N}u_{i_1j_1}\ldots u_{i_kj_k}
=\int_{S_N}\chi\left(\sigma\Big|\sigma(j_1)=i_1,\ldots,\sigma(j_k)=i_k\right)
=\frac{(N-k)!}{N!}$$

$(4)\implies(5)$ This is trivial.

$(5)\implies(6)$ This is trivial too.

$(6)\implies(1)$ This is clear, because if $u_{i_1j_1}\ldots u_{i_kj_k}=\chi(\sigma|\sigma(j_1)=i_1,\ldots,\sigma(j_k)=i_k)$ is nonzero, we can find an element $\sigma\in G$ such that $\sigma(j_s)=i_s$, $\forall s$.
\end{proof}

Summarizing, we have now a complete picture of the notion of $k$-transitivity for the usual permutation groups $G\subset S_N$, from a probabilistic viewpoint.

\section{Quantum transitivity}

In this section we discuss the quantum permutation group analogue of the above results. Once again, we generally refer to \cite{bco} for the needed preliminary material. We will need as well the general theory from \cite{bi2}, \cite{lmr} which solves the problems at $k=1,2$.

Let us begin with an analogue of Proposition 3.1. The formula here, from \cite{bco}, is:

\begin{proposition}
Consider the quantum permutation group $S_N^+$, with standard coordinates denoted $u_{ij}$. We have then the Weingarten formula:
$$\int_{S_N^+}u_{i_1j_1}\ldots u_{i_kj_k}=\sum_{\pi,\sigma\in NC_k}\delta_\pi(i)\delta_\sigma(j)W_{S_N^+}(\pi,\sigma)$$
In particular, at $k\leq3$, we obtain the same integrals as for $S_N$.
\end{proposition}

\begin{proof}
The formula in the statement is from \cite{bco}, with $NC_k$ being the set of noncrossing partitions of $\{1,\ldots,k\}$, with the Weingarten matrix being given by $W_{S_N^+}=G_{S_N^+}^{-1}$, where $G_{S_N^+}(\pi,\sigma)=N^{|\pi\vee\sigma|}$, and with the formula itself basically coming from the fact that the matrix formed by the integrals in the statement is the projection onto $Fix(u^{\otimes k})$.

Regarding now the second assertion, let us recall as well from \cite{bco} that it is possible to write a Weingarten formula for the usual symmetric group $S_N$ as well, as follows:
$$\int_{S_N}u_{i_1j_1}\ldots u_{i_kj_k}=\sum_{\pi,\sigma\in P_k}\delta_\pi(i)\delta_\sigma(j)W_{S_N}(\pi,\sigma)$$

Here $P_k$ is the set of all partitions of $\{1,\ldots,k\}$, and the Weingarten matrix is given by $W_{S_N}=G_{S_N}^{-1}$, with $G_{S_N}(\pi,\sigma)=N^{|\pi\vee\sigma|}$. This formula is of course not very interesting, because the integrals on the right are already known, from Proposition 3.1 above.

However, in our context, we can use this formula, by comparing it with the one in the statement. Since at $k\leq 3$ all the partitions of $\{1,\ldots,k\}$ are noncrossing, we have $P_k=NC_k$, the Weingarten functions coincide as well, and we obtain the result.
\end{proof}

Regarding now the notion of $k$-transitivity, there are some changes here as well. Each magic unitary matrix $u=(u_{ij})$ produces a corepresentation $u^{\otimes k}=(u_{i_1j_1}\ldots u_{i_kj_k})$, and so a coaction map $\Phi:(\mathbb C^N)^{\otimes k}\to C(G)\otimes(\mathbb C^N)^{\otimes k}$, given by the following formula:
$$\Phi(e_{i_1\ldots i_k})=\sum_{j_1\ldots j_k}u_{i_1j_1}\ldots u_{i_kj_k}\otimes e_{j_1\ldots j_k}$$

The problem is that $span(I_N^k)$ is no longer invariant, due to the fact that the variables $u_{ij}$ no longer commute. We can only say that $span(J_N^k)$ is invariant, where:
$$J_N^k=\left\{(i_1,\ldots,i_k)\in\{1,\ldots,N\}^k\Big|i_1\neq i_2\neq\ldots\neq i_k\right\}$$

Indeed, by using the fact, coming from the magic condition on $u$, that  $a\neq c,b=d$ implies $u_{ab}u_{cd}=0$, we obtain that for $i\in J_N^k$ we have, as desired:
$$\Phi(e_{i_1\ldots i_k})=\sum_{j_1\neq j_2\neq\ldots\neq j_k}u_{i_1j_1}\ldots u_{i_kj_k}\otimes e_{j_1\ldots j_k}$$

We can study the transitivity properties of this coaction, as follows:

\begin{proposition}
Given a closed subgroup $G\subset S_N^+$, consider the associated coaction $\Phi:span(J_N^k)\to C(G)\otimes span(J_N^k)$. The following conditions are then equivalent:
\begin{enumerate}
\item $Fix(\Phi)=\{\xi|\Phi(\xi)=1\otimes\xi\}$ is $1$-dimensional.

\item $Fix(\Phi)=\mathbb C\eta$, where $\eta=\sum_{i\in J_N^k}e_{i_1\ldots i_k}$.

\item $\sum_{i\in J_N^k}\int_Gu_{i_1i_1}\ldots u_{i_ki_k}=1$.
\end{enumerate}
If these conditions are satisfied, we say that the coaction $\Phi$ is transitive.
\end{proposition}

\begin{proof}
In order to prove $(1)\iff(2)$, we just have to check that we have indeed $\Phi(\eta)=1\otimes\eta$, with $\eta$ being as in the statement. By definition of $\Phi$, we have:
$$\Phi(\eta)=\sum_{j_1\neq j_2\neq\ldots\neq j_k}\sum_{i_1\neq i_2\neq\ldots\neq i_k}u_{i_1j_1}\ldots u_{i_kj_k}\otimes e_{j_1\ldots j_k}$$

Let us compute the middle sum $S$. When summing over indices $i_1\neq i_2$ we obtain $(1-u_{i_2j_1})u_{i_2j_2}\ldots u_{i_kj_k}=u_{i_2j_2}\ldots u_{i_kj_k}$, then when summing over indices $i_2\neq i_3$ we obtain $(1-u_{i_3j_2})u_{i_3j_3}\ldots u_{i_kj_k}=u_{i_3j_3}\ldots u_{i_kj_k}$, and so on, up to obtaining $\sum_{j_k}u_{j_kj_k}=1$ at the end. Thus we have $S=1$, and so the condition $\Phi(\eta)=1\otimes\eta$ is satisfied indeed.

Regarding now $(1)\iff(3)$, this comes from the general formula $\dim Fix(\Phi)=\int_G\chi$, where $\chi$ is the character of the corepresentation associated to $\Phi$. Indeed, in the standard basis $\{e_i|i\in J_N^k\}$ we have $\chi=\sum_{i\in J_N^k}u_{i_1i_1}\ldots u_{i_ki_k}$, and this gives the result. 
\end{proof}

We have the following partial analogue of Theorem 3.2 above:

\begin{proposition}
Given a closed subgroup $G\subset S_N^+$, with $N\geq4$, with matrix coordinates denoted $u_{ij}$, and a number $k\in\mathbb N$, the following conditions are equivalent:
\begin{enumerate}
\item The action $G\curvearrowright span(J_N^k)$ is transitive.

\item $Fix(u^{\otimes k})$ is minimal, i.e. is the same as for $G=S_N^+$.

\item $\dim Fix(u^{\otimes k})=C_k$, where $C_k$ is the $k$-th Catalan number.

\item $\int_Gu_{i_1j_1}\ldots u_{i_kj_k}$ is the same as for $G=S_N^+$, for any $i,j\in J_N^k$.
\end{enumerate}
\end{proposition}

\begin{proof}
This follows as in the first part of the proof of Theorem 3.2, by performing changes where needed, and by using the general theory from \cite{bco}, as an input:

$(1)\implies(2)$ This follows from the fact that $u^{\otimes k}$ comes by summing certain actions $G\curvearrowright J_N^r$ with $r\leq k$, and the transitivity at $k$ implies the transitivity at any $r\leq k$.

$(2)\implies(3)$ This comes from the well-known fact that for the quantum group $S_N^+$ with $N\geq4$, the multiplicity $\#(1\in u^{\otimes k})$ equals the $k$-th Catalan number $C_k$.

$(3)\implies(4)$ This comes from the well-known fact that $P_{i_1\ldots i_k,j_1\ldots j_k}=\int_Gu_{i_1j_1}\ldots u_{i_kj_k}$ is the orthogonal projection onto $Fix(u^{\otimes k})$, coming from \cite{wo1}.

$(4)\implies(1)$ This follows by taking $i=j$ and then summing over this index, by using the transitivity criterion for $G\curvearrowright span(J_N^k)$ from Proposition 4.2 (3).
\end{proof}

Now let us compare Theorem 3.2 and Proposition 4.3, with some input from Proposition 4.1 as well. We conclude that the notion of $k$-transitivity for the subgroups $G\subset S_N$ extends to the quantum group case, $G\subset S_N^+$, depending on the value of $k$, as follows:

(1) At $k=1,2$ everything extends well, due to the results in \cite{bi2}, \cite{lmr}.

(2) At $k=3$ we have a good phenomenon, $P_3=NC_3$, and a bad one, $I_N^3\neq J_N^3$.

(3) At $k\geq4$ we have two bad phenomena, namely $P_k\neq NC_k$ and $I_N^k\neq J_N^k$.

Summarizing, our study suggests that things basically stop at $k=3$. So, as a conclusion, let us record the definition and main properties of the $3$-transitivity:

\begin{theorem}
A closed subgroup $G\subset S_N^+$ is $3$-transitive, in the sense that we have $\dim(Fix(u^{\otimes 3}))=5$, if and only if, for any $i,k,p$ distinct and any $j,l,q$ distinct:
$$\int_Gu_{ij}u_{kl}u_{pq}=\frac{1}{N(N-1)(N-2)}$$
In addition, in the classical case, we recover in this way the usual notion of $3$-transitivity.
\end{theorem}

\begin{proof}
We know from Proposition 4.3 that the 3-transitivity condition is equivalent to the fact that the integrals of type $\int_Gu_{ij}u_{kl}u_{pq}$ with $i\neq k\neq p$ and $j\neq l\neq q$ have the same values as those for $S_N^+$. But these latter values are computed by Proposition 3.1 and Proposition 4.1, and the 3-transitivity condition follows to be equivalent to:
$$\int_Gu_{ij}u_{kl}u_{pq}=
\begin{cases}
\frac{1}{N(N-1)(N-2)}&{\rm if}\ \ker(ikp)=\ker(jlq)=|||\\
\frac{1}{N(N-1)}&{\rm if}\ \ker(ikp)=\ker(jlq)=\sqcap\hskip-4.55mm{\ }_{{\ }^|}\\
0&{\rm if}\ \{\ker(ikp),\ker(jlq)\}=\{|||,\sqcap\hskip-4.55mm{\ }_{{\ }^|}\,\}
\end{cases}$$

Now observe that the last formula is automatic, by using the traciality of the integral and the magic assumption on $u$, and that the middle formula follows from the first one, by summing over $i,j$. Thus we have are left with the first formula, as stated. 

Finally, the last assertion follows from Theorem 3.2, applied at $k=3$.
\end{proof}

There are of course many interesting questions left. In the general case, $k\in\mathbb N$, the situation is quite unclear, and a first question is whether Proposition 4.3 can be completed or not with two supplementary conditions, in the spirit of (5) and (6) in Theorem 3.2. Also, in the case $k=3$ we have the question of fine-tuning our above analytic definition, which is quite abrupt, with a full algebraic study, in the spirit of \cite{bi2}, \cite{lmr}.

In what follows we will focus on some related matrix modelling questions, where the analytic results that we have so far are precisely those that we need.

\section{Matrix models}

We recall from \cite{bne} that a matrix model $\pi:C(G)\to M_K(C(X))$ is called flat when the images $P_{ij}=\pi(u_{ij})$ have constant rank. The terminology here comes from the fact that the bistochastic matrix $T_{ij}=tr(P_{ij})$ must be the flat matrix, $T=(1/N)_{ij}$.

There is an obvious relation here with the notion of transitivity, coming from:

\begin{proposition}
A matrix model $\pi:C(G)\to M_K(C(X))$ is flat precisely when  
$$tr(P_{ij})=\int_Gu_{ij}=\frac{1}{N}$$
holds, for any $i,j$. In particular, if such a $\pi$ exists, $G\subset S_N^+$ must be transitive.
\end{proposition}

\begin{proof}
This is indeed clear from Proposition 4.3, applied at $k=1$. See \cite{bch}.
\end{proof}

Summarizing, the notion of flatness naturally comes from the notion of transitivity, in its $\int_Gu_{ij}=\frac{1}{N}$ formulation. Based on this observation, we can now formulate:

\begin{definition}
A matrix model $\pi:C(G)\to M_K(C(X))$ is called doubly flat when 
$$tr(P_{ij}P_{kl})=
\begin{cases}
\frac{1}{N}&{\rm if}\ i=k,j=l\\
0&{\rm if}\ i=k,j\neq l\ {\rm or}\ i\neq k,j=l\\
\frac{1}{N(N-1)}&{\rm if}\ i\neq k,j\neq l
\end{cases}$$
holds for any $i,j,k,l$.
\end{definition}

Observe that, due to Proposition 5.1, a doubly flat model must be flat. In analogy with the last assertion in Proposition 5.1, we will see in what follows that the existence of a doubly flat model implies that the quantum group must be doubly transitive.

We can talk as well about triple flatness, as follows:

\begin{definition}
A matrix model $\pi:C(G)\to M_K(C(X))$ is called triple flat when 
$$tr(P_{ij}P_{kl}P_{pq})=\frac{1}{N(N-1)(N-2)}$$
holds for any $i,k,p$ distinct and any $j,l,q$ distinct.
\end{definition}

Here we have used the formula from Theorem 4.4 above. It is possible of course to enlarge Theorem 4.4, and the above definition as well, with a more precise formula, containing the values of all the possible $5\times 5=25$ types of integrals which can appear. We will not need this here, the supplementary formulae being anyway corollaries.

As a main theoretical result regarding these notions, we have:

\begin{theorem}
Assuming that $C(G)$ has a doubly flat model, $G\subset S_N^+$ must be doubly transitive. The same happens for the triple flatness and transitivity.
\end{theorem}

\begin{proof}
This follows by using the Ces\`aro limiting formula for the Haar functional of the Hopf image, coming from the work in \cite{bfs}, \cite{fsk}, \cite{sso}, \cite{wa3}. Indeed, at $k=2$, the values of the Haar functional on the variables $u_{ij}u_{kl}$ appear by performing a Ces\`aro construction to the matrix in the statement, and we are led to the same values. Thus the subgroup $G'\subset G$ which produces the Hopf image must be doubly transitive, and this implies that $G$ itself is doubly transitive, and we are done. At $k=3$ the proof is similar.
\end{proof}

At a more concrete level now, we have several interesting examples of transitive subgroups $G\subset S_N^+$, coming from the complex Hadamard matrices \cite{bbs}, \cite{tzy}, and from the Weyl matrices \cite{ban}, \cite{bne}. Skipping here the definitions and main properties of these quantum groups, for which we refer to \cite{bch}, the double transitivity result is as follows:

\begin{proposition}
We have the following results:
\begin{enumerate}
\item The quantum group $G\subset S_N^+$ associated to an Hadamard matrix $H\in M_N(\mathbb C)$ is doubly transitive precisely when the profile graph of $H$ is connected.

\item The quantum groups $G\subset S_{N^2}^+$ associated to the standard Weyl matrices, coming from $\mathbb Z_N\times\mathbb Z_N$ with Fourier cocycle, are doubly transitive.
\end{enumerate}
\end{proposition}

\begin{proof}
As already mentioned, we refer to \cite{bch} for the general theory of the above quantum groups. With standard terminology and notations, the proof is as follows:

(1) The result here is well-known in subfactor theory \cite{jsu}, and in the quantum group case, this can be deduced directly as well, by using the theory in \cite{bbs}.

(2) This follows by using Theorem 5.4 above, which tells us that we must simply check the double flatness of the model, and from the computations in \cite{ban}, \cite{bne}.
\end{proof}

Regarding the triple transitivity questions, the problems are open here. Open as well is the question of defining a notion of ``double quasi-flatness'', in relation with \cite{bfr}. 

\section{Minimal flatness}

We recall from the beginning of section 5 that a matrix model $\pi:C(G)\to M_K(C(X))$ is by definition flat when the images $P_{ij}=\pi(u_{ij})$ have constant rank. The simplest situation is when this common rank is $R=1$, and this ``minimality'' assumption was in fact often made in the recent literature \cite{bch}, \cite{bfr}, without special mention.

In order to comment now on what happens in the context of the doubly flat models, let us being with the following result, which is standard:

\begin{proposition}
Given a doubly transitive subgroup $G\subset S_N^+$, and a number $K\in\mathbb N$, we have a well-defined universal doubly flat model,
$$\pi:C(G)\to M_K(C(X))\quad,\quad u_{ij}\to P_{ij}^x$$
obtained by imposing the Tannakian conditions which define $G\subset S_N^+$. The same happens for the triple flatness and transitivity.
\end{proposition}

\begin{proof}
This follows indeed as in the transitive case, by using Woronowicz's Tannakian duality \cite{wo2}, and more specifically, its formulation from \cite{mal}. To be more precise, for $G=S_N^+$ the model space is a certain algebraic manifold, obtained by imposing the single, double or triple flatness condition, depending on the statement which is to be proved. 

In the general case now, where $G\subset S_N^+$ is arbitrary, what we have to do is to further impose the following Tannakian conditions, which define $G\subset S_N^+$:
$$T\in Hom(P^{\otimes k},P^{\otimes l})\quad,\quad \forall\,T\in Hom(u^{\otimes k},u^{\otimes l})$$

Thus, the model space is indeed well-defined, as a certain compact space, appearing as a subalgebraic manifold of the model space for $S_N^+$, and this gives the result.
\end{proof}

In the classical case, the situation is quite special, because the elements $P_{ij}P_{kl}$ appearing in Definition 5.2 are orthogonal projections. Thus, we have the following result:

\begin{proposition}
Assuming that $\pi:C(G)\to M_K(\mathbb C)$ is a doubly flat model for a doubly transitive group $G\subset S_N$, mapping $u_{ij}\to P_{ij}$, the common rank 
$$R=rank(P_{ij})$$
must satisfy $(N-1)|R$. 
\end{proposition}

\begin{proof}
This follows indeed from the above observation, with the quotient $\frac{R}{N-1}$ being the common rank of the projections $P_{ij}P_{kl}$, with $i\neq k,j\neq l$.
\end{proof}

The simplest non-trivial instance of this divisibility phenomenon appears for the group $G=S_3$, and we have here the following result, coming from \cite{bbs}:

\begin{theorem}
The doubly flat models for $C(S_3)$ which are ``minimal'', in the sense that we have $R=N-1$, appear from magic unitaries of the type
$$u=\begin{pmatrix}
a+b&c+d&e+f\\
c+f&a+e&b+d\\
e+d&b+f&a+c
\end{pmatrix}$$
with $a,b,c,d,e,f$ being rank $1$ projections, summing up to $1$.
\end{theorem}

\begin{proof}
As explained in \cite{bbs}, the representations of $C(S_3)$ must appear from magic matrices as above, with $a,b,c,d,e,f$ being projections which sum up to 1. In the minimal case the common rank must be $R=3-1=2$, and this gives the result. 
\end{proof}

It is elementary to prove, based on this description, that the universal minimal doubly flat model for $C(S_3)$ is stationary, in the sense of \cite{ban}. However, finding analogues of this result for more general subgroups $G\subset S_N$, and especially for non-classical quantum groups $G\subset S_N^+$, is an open problem, that we would like to raise here.

\end{document}